\documentclass[11pt]{amsart}

\usepackage{graphicx}                    
\usepackage{amsmath}                     
\usepackage{amsfonts}                    
\usepackage{amssymb}                     
\usepackage{amsthm}                      
\usepackage[normalem]{ulem}              
\usepackage[noadjust,verbose,sort]{cite} 
\usepackage{setspace}                    
\usepackage{hyperref}
\usepackage{tikz}

\theoremstyle{plain}
\newtheorem{theorem}{Theorem}[subsection]
\newtheorem{proposition}[theorem]{Proposition}
\newtheorem{lemma}[theorem]{Lemma}

\theoremstyle{definition}
\newtheorem{definition}[theorem]{Definition}
\newtheorem*{defn*}{Definition}
\newtheorem*{question*}{Question}

\newtheorem{example}[theorem]{Example}
\newtheorem*{example*}{Example}
\newtheorem{remark}[theorem]{Remark}
\newtheorem{algorithm}[theorem]{Algorithm}
\title{The Survival Complex}
\author{Anna-Rose G. Wolff }
\date{\today}
\begin{document}

\begin{abstract}
We introduce a new way to associate a simplicial complex called the \emph{survival complex} to a commutative semigroup with zero. Restricting our attention to the semigroup of monomials arising from an Artinian monomial ring, we determine that any such complex has an isolated point.  Indeed, we show that there is exactly one isolated point essentially only in the case where the monomial ideal is generated purely by powers of the variables.  This allows us to recover Beintema's result that an Artinian monomial ring is Gorenstein if and only if it is a complete intersection.  A key ingredient of the translation between the pure power result and Beintema's result is given by the one-to-one correspondence we show between the so-called \emph{truly isolated} points of our complex and the generators of the socle of the defining ideal.  In another relation between the geometry of the complex and the algebra of the ring, we essentially give a correspondence between the nontrivial connected components of the complex and the factors of a fibre product representation of the ring. Finally, we explore algorithms for building survival complexes from specified isolated points.  That is, we work to build the ring out of a description of the socle.
\end{abstract}

\maketitle
\section[Introduction]{Introduction}

\subsection{What Is a Survival Complex}

 Let $(S,\cdot,0)$ be a commutative multiplicative semigroup with an absorbing element $0$. Then one can form an abstract simplicial complex, called the survival complex of $S$, whose vertices are the nonzero zero-divisors of $S$ and where one forms a simplex from $s_1, \dotsc, s_n \in S$ whenever $s_1 \cdots s_n \neq 0$ -- that is, whenever a subset ``survives'' multiplication.\\
 In particular we investigate the survival complex of $S$, denoted by $\Sigma(S)$ when $S$ is as follows. Let $k$ be a field, $X_1, \dotsc, X_t$ variables, $I$  a monomial ideal in $A := k[X_1, \dotsc, X_t]$, and $R := A/I$. Then $S$ consist of the images of the nontrivial monomials of $A$ in $R$, under the inherited multiplication, along with $0$. For this paper, we will only consider ideals $I$ where $R = A/I$ contains only finitely many monomials. In other words, for the purpose of this paper all survival complexes are assumed to be finite.  \\

 \begin{example}
Let $R = \frac{k[x,y,z]}{(x^3,y^3)}$. Then the survival complex $\Sigma (R)$ is composed of all surviving monomials of the ring $R$\\

\end{example}

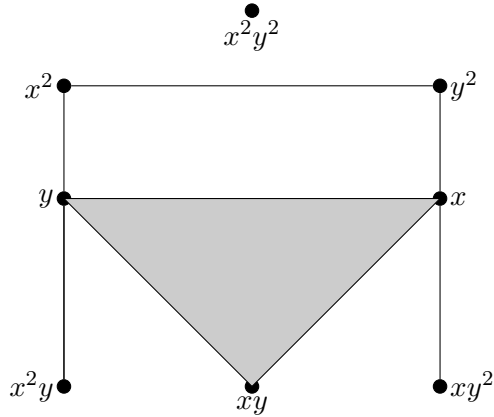
\begin{figure}[h]
\begin{center}

\begin{tikzpicture}
[scale=5, vertices/.style={draw, fill=black, circle, inner sep=1.8pt}]

\node[vertices] (y) at (-0.5,0) {};
\node[left] at (-0.5,0) {$y$};
\node[vertices] (x) at (0.5,0) {};
\node[right] at (0.5,0) {$x$};

\node[vertices] (xy) at (0,-0.5) {};
\node[below] at (0,-0.5) {$xy$};
\node[vertices] (yy) at (0.5,0.3) {};
\node[right] at (0.5,0.3) {$y^2$};
\node[vertices] (xx) at (-0.5,0.3) {};
\node[left] at (-0.5,0.3) {$x^2$};

\node[vertices] (xxy) at (-0.5,-0.5) {};
\node[left] at (-0.5,-0.5) {$x^2y$};
\node[vertices] (xxyy) at (0,0.5) {};
\node[below] at (0,0.5) {$x^2y^2$};
\node[vertices] (xyy) at (0.5,-0.5) {};
\node[right] at (0.5,-0.5) {$xy^2$};

\filldraw[fill=black!20, draw=black] (0,-0.5)--(-0.5,0)--(0.5,0)--cycle;

\draw (xyy)--(x);
\draw (yy)--(x);
\draw (yy)--(xx);
\draw (y)--(xx);
\draw (y)--(xxy);
\draw (y)--(xxy);

\end{tikzpicture}

\caption{Survival Complex of $R = \frac{k[x,y,z]}{(x^3,y^3)}$}
\end{center}
\end{figure}

%
%

\begin{definition} Let an ideal $I$ of the ring $A = k[x_1, \ldots, x_t]$ be generated only by monomials of the form $x_{j}^a$. Then we call the ring $S = \frac{A}{I}$ a \emph{pure power} ring and the survival complex $\Sigma(S)$ a \emph{pure power} survival complex.\\
\end{definition}

For an example of a pure power ring and its associated pure power survival complex, see Figure 1.\\
%
%
%

\subsection{Connection to Zero Divisor Graphs and Stanley Reisner Complexes}
The zero divisor graph of a semigroup is an object which has been studied by multiple authors \cite{Meyer1, Meyer2, ZGSemigroup}. In fact, in \cite{Meyer1} Lisa DeMeyer and Frank DeMeyer introduce a way to associate the zero divisor graph of a semigroup with an abstract simplicial complex. However, the authors of \cite{Meyer1,Meyer2} are primarily interested in discovering when, given an arbitrary graph, that graph is the zero divisor graph of a semigroup.\\
Our work is related to this idea in the following way. Given an arbitrary survival complex $\Sigma (S)$, the dual of the 1-skeleton of $\Sigma (S)$ will be a graph where two monomials $p,q \in S$ are connected precisely when $pq = 0$. In other words, the dual of the 1-skeleton of $\Sigma (S)$ is the zero divisor graph of the monomial semigroup of $S$.\\
Instead of analyzing arbitrary simplicial complexes in order to determine when there exists an $S$ such that $\Sigma (S)$ is the complex of $S$, we are more interested in what $\Sigma (S)$ tells us about the algebraic structure of $S$. While our approach is combinatiorial in nature, the structure of $\Sigma (S)$ allows us to draw conclusions about such algebraic properties as the dimension of the socle of $S$.\\
There are also similarities between the structure of the survival complex and the structure of Stanley-Reisner complexes. While we do not go into the similarities and differences in this thesis, we would like to note that given a squarefree monomial ideal $I \subset k[x_1, \ldots, x_n]$ the Stanley-Reisner complex of $\frac{k[x_1, \ldots, x_n]}{I}$ is an induced subcomplex of the survival complex $\Sigma \left(\frac{k[x_1, \ldots, x_n]}{I + (x_{1}^2, \ldots, x_{n}^2)} \right)$. The reader who wishes to learn more about Stanley-Reisner complexes is advised to look at \cite{C-MRings} and \cite{SurveySR}.

\subsection{Basic Properties of Survival Complexes}

An important component of survival complexes is isolated points. A point in a survival complex is \emph{isolated} when there are no edges connecting that point to any other point. This leads us to the following definitions.\\

\begin{definition}
Let $\Sigma (R)$ be the survival complex of $R = \frac{k[x_1, \ldots, x_n]}{I}$. We say that a monomial $m \in \Sigma (R)$ is \emph{truly isolated} if for all $1 \leq i \leq n$, $x_im =0$.\\
\end{definition}

\begin{remark}
The reader should be comfortable with the equivalence between this definition and the statement that $mq = 0$ for all monomials $q \in \Sigma (R)$. If $mq = 0$ for all monomials in $\Sigma (R)$, then trivially for all $1 \leq i \leq n$, $x_im =0$. Furthermore, since every monomial $q \in \Sigma (R)$ is composed of the variables $x_1, \ldots, x_n$, then if for all $1 \leq i \leq n$, $x_im =0$, we automatically get that $mq = 0$ for all monomials $q \in \Sigma (R)$.\\
\end{remark}

\begin{definition}
Let $\Sigma (R)$ be the survival complex of $R = \frac{k[x_1, \ldots, x_n]}{I}$. We say that a variable $x_t \in \Sigma (R)$ is \emph{quasi-isolated} if for all $1 \leq i \leq n$ with $i \neq t$, $x_ix_t =0$, but $x_{t}^2 \neq 0$ and $x_{t}^3 = 0$.\\
\end{definition}

\begin{remark}
Notice that we need not worry about an arbitrary monomial $m \in \Sigma (R)$ in our definition of quasi-isolated. This follows from the fact that for an arbitrary monomial $m \in \Sigma (R)$ if $m^2 \neq 0$, then for some factor $x_i$ of $m$ we have that $x_im \neq 0$. Hence there is an edge from $m$ to $x_i$ making $m$ nonisolated.\\
Quasi-isolated points are precisely those points $x_t$ where there are no edges from $x_t$ to any other $x_i$, but if loops were permitted in $\Sigma (R)$ there would be a loop on $x_t$.\\

\end{remark}

When we simply refer to the isolated points of a survival complex, it is assumed that we mean both the truly isolated and the quasi-isolated points of the survival complex in question.\\

\begin{proposition}
Every finite dimensional survival complex must have at least one isolated point.
\end{proposition}

\begin{proof}
Assume there exists a finite dimensional survival complex $\Sigma (S)$ which does not contain an isolated point.\\
As $\Sigma (S)$ is finite dimensional, it must contain only a finite number of monomials. Since the set of degrees of these monomials is a finite set, there must be a maximum possible degree.\\
Take a monomial $m$ that is of maximum degree among the monomials of $\Sigma (S)$.
Since $m$ is not isolated, there exists some monomial $x  \in \Sigma (S)$ such that $xm \neq 0$.
But then the degree of $xm$ is greater than the degree of $m$.
Therefore the degree of $m$ is not maximal, a contradiction.
Hence every finite dimensional survival complex has at least one isolated point.
\end{proof}

One of the main questions considered in regards to survival complexes was the relationship between the ideal used to create a survival complex $\Sigma (S)$ and the number of isolated points of $\Sigma (S)$.\\


\begin{proposition}
Any pure power survival complex 
contains exactly one truly isolated point.
\end{proposition}

\begin{proof}
 
Let $\Sigma (S)$ be the survival complex of $S = \frac{A}{(x_{1}^{m_1}, \ldots,x_{n}^{m_n})}$ where $A = k[x_{1 }, \ldots, x_n]$ and $m_j \geq 2$ for all $j$.
Then the point $q =x_{1}^{m_1 -1} \dots x_{n}^{m_n -1}$ is in $\Sigma (S)$ and is obviously truly isolated since $x_jm =0$ for all $1 \leq j \leq n$.\\
Assume that $ g = x_{1}^{t_1} \dots x_{m}^{t_m}$ is truly isolated where each $t_j \geq 0$ and that for some $j$, $t_j < m_j - 1$. Then clearly $g$ is a factor of $q$ in $A$ which means that $q = 0$ in $S$. Hence $q$ is in the defining ideal of $S$, a contradiction.\\
Thus every pure power survival complex contains only one truly isolated point.\\
\end{proof}

\begin{proposition}
The pure power survival complex $\Sigma (S)$ generated by $S = \frac{k[x] }{(x^3)}$ is the only pure power survival complex to contain a quasi-isolated point.\\

\end{proposition}

\begin{proof}
Let $S = \frac{k[x] }{(x^3)}$ and consider $\Sigma (S)$. This complex is composed of precisely two points, $x$ and $x^2$. $x$ is multiplied by itself to form $x^2$. However, as we do not allow loops in survival complexes, the result is that $x$ is quasi-isolated.\\

\noindent Now let $S = \frac{k[x_{1 }, \ldots, x_n]}{(x_{1}^{m_1}, \ldots,x_{n}^{m_n})}$ and consider $\Sigma (S)$.\\
Assume that $x_t \in \Sigma (S)$ is a quasi-isolated point.
Since $x_{t}^2$ must be in $\Sigma (S)$ for $x_t$ to be quasi-isolated, $m_t \geq 3$.
Assume that $n >1$. Then there is another monomial $x_j \in \Sigma (S)$.
But since $S$ is a pure power complex, $x_tx_j \in \Sigma (S)$.
Hence $x_t$ is not quasi-isolated.\\ Therefore, $n = 1$.\\

\noindent Thus we have $S = \frac{k[x] }{(x_{t}^{m_t})}$
When $m_t = 3$, we know that $\Sigma (S)$ has a quasi-isolated point so we only need to consider the case when $m_t > 3$.
Then $x_t$, $x_{t}^2$, and $x_{t}^3$ are in $\Sigma (S)$ and $x_{t}^3 = x_tx_{t}^2$.\\
Hence there is an edge from $x_t$ to $x_{t}^2$ which implies that $x_t$ is not quasi-isolated.\\

Therefore, the pure power survival complexes $\Sigma (S)$ generated by $S = \frac{k[x] }{(x^3)}$ is the only pure power survival complex to contain a quasi-isolated point.\\

\end{proof}

By definition every factor of an isolated point is a point in the survival complex. Furthermore, every point in the survival complex corresponds to a simplex in the survival complex. Thus in order to find the largest simplex in a survival complex, you only have to examine the isolated points.\\


\begin{proposition} 
Let $\Sigma (S)$ be the survival complex of $S = \frac{k[x_1, \ldots,x_n]}{I}$. Then $p \neq 0$ is a monomial in $\Sigma(S)$ if and only if some truly isolated point of $\Sigma(S)$ has $p$ as a factor.
\end{proposition}

\begin{proof}
First, note that given $S = \frac{k[x_1, \ldots, x_n]}{I}$, there is a natural poset structure on the monomials of $S$. For $p,q \in S$ simply define $p \leq q$ whenever $p$ is a factor of $q$.\\
Furthermore, since we are only considering rings with a finite number of monomials, this poset is finite.\\

Assume that some truly isolated point $q \in \Sigma (S)$ has $p$ as a factor. Since $q \not \in I$, this implies that $p \not \in I$. Hence $0 \neq p \in S$ and therefore $p \in \Sigma(S)$.\\

Now assume that $p \neq 0$ is a monomial in $\Sigma (S)$.\\
Since $p \in \Sigma (S)$, $p$ is also in $S$. So consider the poset $P$ on the monomials of $S$ defined as above.
Take a maximal chain in $P$ containing $p$ and call the largest element in the chain $m$.\\
Since $m$ is a maximal element of a maximal chain containing $p$, then for all $q \in S$ either $q \leq m$ or $q$ is not comparable to $m$.
But since for all $x_i$, $1 \leq i \leq n$, we have $x_im \geq m$ this implies that $x_im \not \in \Sigma(S)$.
Therefore $x_im \in I$ and hence $x_im = 0 \in S$.
By definition, this means that $m$ is truly isolated in $\Sigma (S)$.\\
Therefore some truly isolated point of $\Sigma (S)$ contains $p$ as a factor.\\

\end{proof}

We also have a nice relationship between containment between two ideals of a ring and the survival complexes generated by those ideals. \\

\begin{proposition}
Consider the survival complexes $\Sigma(S)$ and $\Sigma (T)$ generated by $S = \frac{k[x_{1}, \ldots, x_n]}{I}$ and $T = \frac{k[x_1, \ldots,x_n ]}{J}$ with  $I \subseteq J$. Then $\Sigma(T) \subseteq \Sigma(S)$. This containment is proper if $I \neq J$.
\end{proposition}

\begin{proof}
Choose a point $q \in \Sigma(T)$. Then the monomial associated with $q$ is not in $J$. Hence it is not in $I$. Therefore $q \in \Sigma(S)$.\\ 
Since $ q$ was chosen arbitrarily, this means that $\Sigma(T) \subseteq \Sigma(S)$.\\
Now assume that $I \neq J$. 
Thus there is a monomial in $J$ which is not in $I$, call it $m$. Since $m \in J$ this means that $m \not \in \Sigma(T)$. 
However, $m$ is contained in $k[x_{1}, \ldots, x_n]$ while not being contained in $I$. Thus $m \in \Sigma(S)$.\\
Hence $\Sigma(S) \neq \Sigma(T)$.

\end{proof}

\begin{remark}
The previous proposition implies that every survival complex we consider is a subcomplex of a pure power survival complex due to the fact every ideal $J$ within our domain of inquiry contains the ideal $I$ generated by the pure power generators of $J$.

\end{remark}
\section[Fibre Products, Socles, and More]{ Fibre products, Socles, and More}

\subsection{Fibre Products, Survival Complexes, and Connected Components}

Given two rings $R$ and $S$ where $R = \frac{k[x_1, \ldots, x_n]}{I}$ and $S = \frac{k[y_1, \ldots, y_m]}{J}$ the fibre product $R \circ S \cong \frac{k[x_1, \ldots,x_n,y_1, \ldots, y_m]}{(I,J,x_iy_j)}$ where $x_iy_j$ represents all the pairwise products where $x_i \in \lbrace x_1, \ldots, x_n \rbrace$ and $y_j \in \lbrace y_1, \ldots , y_m \rbrace$ \cite{GIIFP}.

We will now connect fibre products to survival complexes that have more than one nontrivial connected component.

Since every survival complex has at least one isolated point, it follows that every survival complex trivially has at least two connected components.\\

\begin{definition}
 
 A \emph{trivially connected} component is a connected component consisting of only one point.\\
 \end{definition}
 
 \begin{definition}
 
 A \emph{quasi-trivially connected} component is a connected component consisting of a quasi-isolated point of the form $x_t$.\\
 \end{definition}
 
 \begin{remark}
 The survival complex $\Sigma \left(\frac{k[x]}{(x^2)}\right)$ will consist of precisely one point, the monomial $x$. A trivially connected component which is not a quasi-trivially connected component can not be the survival complex of a ring. This distinction will be important later.\\
 \end{remark}

\begin{example} An example of a survival complex with two nontrival connected components and three trivially connected components, none of which are quasi-trivially connected.\\
\end{example}

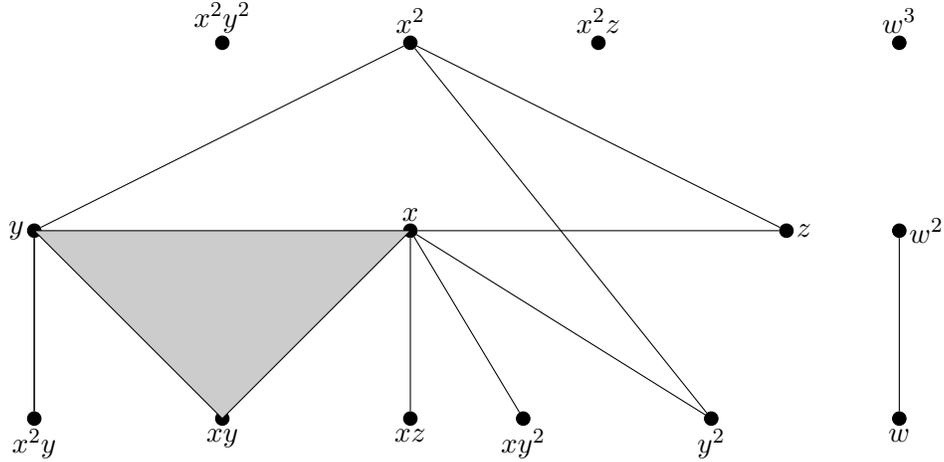
\begin{figure}[h]
\begin{center}

\begin{tikzpicture}
[scale=5, vertices/.style={draw, fill=black, circle, inner sep=1.8pt}]
\node[vertices] (y) at (-1,0) {};
\node[left] at (-1,0) {$y$};
\node[vertices] (x) at (0,0) {};
\node[above] at (0,0) {$x$};
\node[vertices] (z) at (1,0) {};
\node[right] at (1,0) {$z$};
\node[vertices] (xy) at (-0.5,-0.5) {};
\node[below] at (-0.5,-0.5) {$xy$};
\node[vertices] (xz) at (0,-0.5) {};
\node[below] at (0,-0.5) {$xz$};
\node[vertices] (yy) at (0.8,-0.5) {};
\node[below] at (0.8,-0.5) {$y^2$};
\node[vertices] (xx) at (0,0.5) {};
\node[above] at (0,0.5) {$x^2$};
\node[vertices] (xxy) at (-1,-0.5) {};
\node[below] at (-1,-0.5) {$x^2y$};
\node[vertices] (xyy) at (0.3,-0.5) {};
\node[below] at (0.3,-0.5) {$xy^2$};
\node[vertices] (xxyy) at (-0.5,0.5) {};
\node[above] at (-0.5,0.5) {$x^2y^2$};
\node[vertices] (xxz) at (0.5,0.5) {};
\node[above] at (0.5,0.5) {$x^2z$};

\node[vertices] (w) at (1.3,-0.5) {};
\node[below] at (1.3,-0.5) {$w$};
\node[vertices] (ww) at (1.3,0) {};
\node[right] at (1.3,0) {$w^2$};
\node[vertices] (www) at (1.3,0.5) {};
\node[above] at (1.3,0.5) {$w^3$};

\filldraw[fill=black!20, draw=black] (-0.5,-0.5)--(0,0)--(-1,0)--cycle;
\draw (xz)--(x);
\draw (xyy)--(x);
\draw (z)--(x);
\draw (yy)--(x);
\draw (yy)--(xx);
\draw (y)--(xx);
\draw (z)--(xx);
\draw (y)--(xxy);
\draw (y)--(xxy);

\draw (w)--(ww);
\end{tikzpicture}

\caption{Survival Complex of $R = \frac{k[x,y,z,w]}{(x^3,y^3,z^2,w^4,yz,xw,yw,zw)}$}
\end{center}
\end{figure}

\begin{proposition} \label{Components=TIso}
If a survival complex $\Sigma(S)$ contains two nontrivial or quasi-trivially connected components $\mathcal{T}$ and $\mathcal{R}$ then $\mathcal{T}$ and $\mathcal{R}$ must have no variables in common. Furthermore, $\Sigma(S)$ will contain at least two truly isolated points.

\end{proposition}

\begin{proof}

Let $\Sigma(S)$ be a survival complex containing two nontrivial or quasi-trivially connected components, call them $\mathcal{T}$ and $\mathcal{R}$.\\

First assume that one of them is quasi-trivially connected, say $\mathcal{T}$, and that $\mathcal{T}$ and $\mathcal{R}$ both contain the variable $x_i$.
Since $\mathcal{T}$ is a quasi-trivially connected component, $x_i$ must be a quasi-isolated point. But since $\mathcal{R}$ contains the variable $x_i$ as well, then for some monomial $m_1 \in \mathcal{R}$, $x_im_1 \in \mathcal{R}$. Hence there is an edge from $m_1$ to $x_i$ which implies that $x_i$ is not quasi-isolated, a contradiction.\\

Now assume that neither $\mathcal{T}$ or $\mathcal{R}$ is a quasi-trivially connected component and
assume that $\mathcal{T}$ and $\mathcal{R}$ both contain the variable $x_i$. Then for some monomial $m_1 \in \mathcal{R}$, $x_im_1 \in \mathcal{R}$.
Hence there is an edge from $m_1$ to $x_i$.\\ 
Furthermore, for some $m_2 \in \mathcal{T}$, $x_im_2 \in \mathcal{T}$ which implies there is an edge from $m_2$ to $x_i$.\\
Therefore there is a path from $m_1$ to $m_2$ making $\mathcal{T}$ and $\mathcal{R}$ connected, a contradiction.\\
Thus if $\mathcal{T}$ and $\mathcal{R}$ are two nontrivial or quasi-trivially connected components, they must have no variables in common.\\

From Proposition 4 we know that every point in a survival complex must be a factor of a truly isolated point and that every factor of a truly isolated point must be a point in the survival complex.\\
If a truly isolated point contains variables from both $\mathcal{T}$ and $\mathcal{R}$ then there must exist an $m \in \Sigma(S)$ which contains variables from both $\mathcal{T}$ and $\mathcal{R}$. \\
By definition if $m \in \Sigma (S)$ contains both variables from both $\mathcal{T}$ and $\mathcal{R}$ then a subcomplex of $\Sigma (S)$ contains a variable from $\mathcal{T}$ and a variable from $\mathcal{R}$. Hence we are able to create a path from $\mathcal{T}$ to $\mathcal{R}$ using the edges of the subcomplex, a contradiction.\\

Therefore, since every point must be a factor of a truly isolated point and no truly isolated point can contain variables from both $\mathcal{T}$ and $\mathcal{R}$, $\Sigma(S)$ must contain at least two truly isolated points.\\

\end{proof}

\begin{proposition} \label{Components-pairwise}
The survival complex $\Sigma(S)$ of $S = \frac{k[x_1, \ldots, x_n]}{I}$ will contain $m$ nontrivial or quasi-trivially connected components if and only if $I$ contains all pairwise products of the variables in the distinct components.
\end{proposition}

\begin{proof}
Let $\Sigma(S)$ be the survival complex of $S= \frac{k[x_1, \ldots, x_n]}{I}$.\\
Assume that the variables of $S$ can be partitioned into $m$ sets such that for any two sets, the pairwise products of the variables are in $I$.
Let two such variable sets be $\lbrace x_1, \ldots, x_p \rbrace = \mathcal{T}$ and  $ \lbrace x_q, \ldots, x_l \rbrace = \mathcal{R}$.\\
 Then for any $x_i \in \mathcal{T}$ and $x_j \in \mathcal{R}$ no point in the survival complex may contain the factor $x_ix_j$. Thus there are no edges connecting $x_i$ and $x_j$, which means that the monomials generated from $\mathcal{T}$ and $\mathcal{R}$ form two disjoint connected components in $\Sigma(S)$.\\
 But since $\mathcal{T}$ and $\mathcal{R}$ were chosen arbitrarily, this implies that the variables in every such set forms a connected component in $\Sigma (S)$.\\
 Therefore $\Sigma(S)$ contains $m$ nontrivial or quasi-trivially connected components.\\

Now assume that $\Sigma (S)$ has $m$ nontrivial or quasi-trivially connected components, 
two of which are denoted by $\mathcal{T}$ and $\mathcal{R}$.\\
By Proposition~\ref{Components=TIso}, $\mathcal{T}$ and $\mathcal{R}$ have no variables in common implying that no variable in $\mathcal{R}$ has an edge leading to any variable in $\mathcal{T}$. Since each edge corresponds to a point in $\Sigma(S)$, this implies that all pairwise products $x_ix_j$ where $x_i \in \mathcal{R}$ and $x_j \in \mathcal{T}$ must be in the ideal.\\
 Therefore $I$ contains all pairwise products of the variables in the distinct components.\\
\end{proof}

\begin{remark}\label{Fibre-induct}
The reader should note that by using simple induction on the fact that the fibre product of $R = \frac{k[x_1, \ldots, x_n]}{I}$ and $S = \frac{k[y_1, \ldots, y_m]}{J}$ is $R \circ S =\frac{k[x_1, \ldots,x_n,y_1, \ldots, y_m]}{(I,J,x_iy_j)}$, one sees that for rings of the form $A_t = \frac{k[x_{(t,1)} , \ldots, x_{(t,z_t)}]}{I_t}$, $A_1 \circ A_2 \circ \cdots \circ A_p \cong \frac{k[x_{(1,1)} , \ldots, x_{(1,z_1)}, \ldots,x_{(p,1)} , \ldots, x_{(p,z_p)} ]}{(I_1,I_2, \ldots, I_p, x_{(j,v)}x_{(d,b)})}$ where $x_{(j,v)}x_{(d,b)}$ means all the pairwise products of the variables from $A_1, \ldots, A_p$ where $j \neq d$.\\
\end{remark}

\begin{proposition} \label{Components-fibre}

Let $\Sigma (S)$ be the survival complex of $S = \frac{k[x_1, \ldots, x_n]}{I}$. If $\Sigma (S)$ contains $m$ nontrivial or quasi-trivially connected components, then $S$ can be written as the fibre product of $m$ rings, $S = R_1 \circ \cdots \circ R_m$ and $\Sigma (S)$ is the disjoint union of $\Sigma (R_1), \ldots, \Sigma (R_m)$.

\end{proposition}

\begin{proof}

The case for $m =1$ is trivial.\\

Assume that $\Sigma (S)$ has $m =2$ nontrivial or quasi-trivially connected components $\mathcal{U}$ and $\mathcal{V}$.\\
Then $\mathcal{U}$ and $\mathcal{V}$ contain no variables in common. Furthermore, from Proposition~\ref{Components-pairwise} we know that $I$ must contain all pairwise products of variables in $\mathcal{U}$ and $\mathcal{V}$.\\ 
Hence the generators of $I$ can be broken down into three sets: the generators which contain only variables in $\mathcal{U}$, the generators which contain only variables of $\mathcal{V}$, and the pairwise products of variables in $\mathcal{U}$ and $\mathcal{V}$.\\

Therefore we can write $ S =\frac{k[x_1, \ldots,x_a,x_{a+1}, \ldots, x_n]}{(Q,J,x_iy_j)}$ where $\lbrace x_1, \ldots, x_a \rbrace$ are the variables in $\mathcal{U}$, $\lbrace x_{a+1}, \ldots, x_n \rbrace$ are the variables in $\mathcal{V}$, $Q$ is the ideal generated by the generators of $I$ which only contain the variables $x_i$, $1 \leq i \leq a$, and $J$ is the ideal generated by the generators of $I$  which only contain the variables $x_j$, $a+1 \leq j \leq n$.\\
Thus we can write $S = \frac{k[x_1, \ldots, x_a]}{Q} \circ \frac{k[x_{a+1}, \ldots, x_n]}{J}$ by \cite{GIIFP}.\\

Now let $R = \frac{k[x_1, \ldots, x_a]}{Q}$ and $T = \frac{k[x_{a+1}, \ldots, x_n]}{J}$ and consider $\Sigma (R)$ and $\Sigma (T)$.\\
Take an arbitrary $p \in R$. Then $p \in S$ and $p$ is generated only by variables from $\mathcal{U}$. Hence $p$ is either in the connected component generated by the variables in $\mathcal{U}$ or $p$ is an isolated point generated by the variables in $\mathcal{U}$.\\
Now take $0 \neq p \in S$ which is generated by the variables in $\mathcal{U}$.\\
Then by the definition of $R$, $0 \neq p \in R$ which implies that $p \in \Sigma (R)$.\\
Thus there is a correspondence between $\Sigma (R)$ and the subcomplex of $\Sigma(S)$ generated by the variables of $\mathcal{U}$.\\
Following the proof structure in the above paragraph, one also sees that there is a correspondence between $\Sigma (T)$ and the subcomplex of $\Sigma(S)$ generated by the the variables of $\mathcal{V}$.\\
Since every point in $\Sigma (S)$ is either generated by the variables of $\mathcal{U}$ or the variables of $\mathcal{V}$, $\Sigma(S)$ is the disjoint union of the subcomplex generated by the variables of $\mathcal{U}$ and the subcomplex generated by the variables of $\mathcal{V}$.
So by the above correspondence between $\mathcal{U}$ and $\Sigma (R)$ and the correspondence between $\mathcal{V}$ and $\Sigma (T)$ we get that $\Sigma (S)$ is a disjoint union of $\Sigma (R)$ and $\Sigma (T)$.

Therefore we can write $S = R \circ T$ for some rings $R$ and $T$ and $\Sigma (S)$ is the disjoint union of $\Sigma(R)$ and $\Sigma (T)$.\\


Now assume that $S$ has $m >2$ components.\\
By Proposition~\ref{Components=TIso} we know that no two nontrivial or quasi-trivially connected component have any variables in common. Therefore we can write $S = \frac{k[x_{(1,1)}, \ldots, x_{(1,f_1)}, \ldots, x_{(m,1)}, \ldots, x_{(m,f_m)}]}{I}$ where $\lbrace x_{(j,1)} \ldots, x_{(j,f_j)} \rbrace$ are the variables in the jth nontrivial or quasi-trivially connected component of $\Sigma (S)$.\\

By Proposition~\ref{Components-pairwise}, we can then write the ideal $I$ as $I = (J_1,\ldots,J_m,x_{(j,p)}x_{(d,b)})$ where $J_i$ is the ideal generated by the generators of $I$ which only contain the variables in the ith component and $x_{(j,p)}x_{(d,b)}$ indicates all the pairwise products of of variables in different components.\\
 
Therefore $S = \frac{k[x_{(1,1)}, \ldots, x_{(1,f_1)}]}{J_1} \circ \cdots \circ \frac{k[x_{(m,1)}, \ldots, x_{(m,f_m)}]}{J_m}$ as in Remark~\ref{Fibre-induct}. By following the argument presented above, we know that for each $i$ where $1 \leq i \leq m$, $\Sigma \left( \frac{k[x_{(i,1)}, \ldots, x_{(i,f_i)}]}{J_i} \right)$ corresponds to a subcomplex of $\Sigma (S)$ and $\Sigma (S)$ is the disjoint union of all the $\Sigma \left( \frac{k[x_{(i,1)}, \ldots, x_{(i,f_i)}]}{J_i} \right)$.

\end{proof}

\subsection{Socle and Survival Complex}
If $R$ is a local ring with maximal ideal $m$, the socle of $R$ is defined as $\text{soc} R = \lbrace x \in R : xm = 0 \rbrace$ (see~\cite[Def. 1.2.18]{C-MRings}). On the other hand, if $R$ is a standard graded ring with homogeneous maximal ideal $P$ we define the graded socle of $R$ as ${ }^*\text{soc}_R R = \lbrace x \in R : Px = 0 \rbrace $ (see~\cite{LCMIDF}).
Since we are only considering rings of the form $R = \frac{k[x_1, \ldots, x_n]}{I}$, where $I$ is an ideal that contains powers of all the variables, our $R$ is in fact a local ring with only one maximal ideal, $(x_1, \ldots,x_n)$. Hence the socle of $R$ is the same as the graded socle of $R$ which allows us to refer to the socle of $R$ while using the definition of the graded socle.\\
Furthermore, the socle forms a vector space over $R/m$, the basis of which is the set of monomials of $R$ contained in the socle.\\



This in turn allows us to describe the number of truly isolated points of a survival complex $\Sigma(S)$ in regards to the dimension of the socle.\\

\begin{proposition}
Let $\Sigma (R)$ be the survival complex of $R = \frac{k[x_1, \ldots, x_n]}{I}$. Then the truly isolated points of $\Sigma(R)$ are exactly those monomials which generate the socle of $R$.
\end{proposition}

\begin{proof}
The monomials of $\Sigma (R)$ which are truly isolated are precisely the monomials $m$ of $R$ where $mx_j = 0$ for all $1 \leq j \leq n$.\\
The homogeneous maximal ideal of $R$ is the ideal $P = (x_1, \ldots, x_n)$.\\
So for every truly isolated monomial $m$, $mP = 0$. Therefore the truly isolated points of $\Sigma (R)$ correspond to monomials in the socle of $R$.\\
Now assume that some other monomial $q \in R$ is also an element of the socle of $R$. Then $q \neq 0$ which would imply that $q \in \Sigma(R)$.\\
But for all $x_i$, $1 \leq i \leq n$, $qx_i = 0$. Hence $q$ is truly isolated in $\Sigma (R)$.\\
So all monomials in the socle of $R$ are truly isolated points in $\Sigma (R)$.\\
However, the monomials in the socle of $R$ generate the the socle of $R$ as a vector space.\\
Since the truly isolated points of $\Sigma (R)$ are all the monomials of $R$ which are in the socle, the truly isolated points of $\Sigma (R)$ generate the socle of $R$.\\

\end{proof}

\begin{remark} \label{NumTIsopoints-dimsoc}
The minimal number of monomials needed to generate the socle of such a ring $R$ is the dimension of the socle of $R$. Since the the truly isolated points minimally generate the socle, the number of truly isolated points in $\Sigma(R)$ is the dimension of the socle of $R$.\\
\end{remark}

It is a known characterization that an Artinian local ring is Gorenstein if and only if it has socle dimension one (see \cite[Theorem 3.2.10]{C-MRings}).
 In fact, there is a direct correlation between pure power survival complexes and Gorenstein rings.\\

\begin{theorem} \label{Gore-oneSoc}
$S = \frac{k[x_1, \ldots, x_n]}{I}$ is a Gorenstein ring if and only if the survival complex $\Sigma(S)$ is a pure power complex.
\end{theorem}

\begin{proof}
Assume that $\Sigma (S)$ is a pure power survival complex.\\

In Proposition 2 we showed that all pure power complexes have exactly one truly isolated point.  Since the number of truly isolated points correspond to the dimension of the socle of a ring, this implies that every pure power survival complex comes from a Gorenstein ring. \\

Assume that $S =\frac{ k[x_1, \ldots,x_n]}{I} $ is Gorenstein.
Thus it has socle dimension one and therefore $\Sigma (S)$ must have only one truly isolated point (see Remark~\ref{NumTIsopoints-dimsoc}).\\

Because $\Sigma (S)$ has only one truly isolated point $\Sigma(S)$ must have only one nontrivial connected component. If $\Sigma(S)$ had two nontrivial isolated components each connected component would have its own isolated point by Proposition \ref{Components=TIso}.\\
Since every point must appear as a factor in the isolated point by Proposition 4, the truly isolated point of $\Sigma(S)$ must be of the form $w = x_{1}^{m_1} \cdots x_{n}^{m_n}$.\\

Then for any $x_{p}^k$ where $x_{p}^{k+1}$ is $I$, $x_{p}^k$ must connect to at least one other distinct point. Therefore, $x_{p}^k$ must be in a factor of a truly isolated point as well. But since $w$ is the only truly isolated point, it must contain $x_{p}^k$ as a factor. Therefore, $w$ must contain as a factor all the variables of $\Sigma(S)$ and each variable must be able to divide $w$ $l$ times where said variable to the $l+1$st power is in the ideal.\\

Therefore $I$ must be generated by monomials of the form $x_{p}^l$. If the generators of $I$ contained any monomials not divisible by some $x_{p}^l$, then that monomial would be a factor of the truly isolated point of $\Sigma(S)$, a contradiction.\\

Hence $\Sigma(S)$ must be a pure power complex.\\
\end{proof}

\begin{remark}
In his proposition in \cite{NAGADM}, Beintema proved that a ring $R = \frac{k[x_1, \ldots, x_n] }{I}$, where $I$ is a monomial ideal of height $n$, is Gorenstein if and only if $I$ is a complete intersection. Using the survival complex structure we have developed thus far, we offer a different proof of Beintema's theorem below.

\end{remark}

\begin{theorem}[Beintema \cite{NAGADM}]

Let $A = k\left[X_1, \ldots, X_n\right]$, and let $I \subset A$ be an ideal of height $n$ which is generated by monomials. Then $A/I$ is Gorenstein if and only if $I $ is a complete intersection.\\

\end{theorem}

\begin{proof}
First, note that an ideal of height $n$ is a complete intersection if it can be generated by precisely $n$ elements. 

In any polynomial ring $k[x_1, \ldots, x_n]$ an ideal $I = (x_{1}^{a_1}, \ldots, x_{n}^{a_n})$ is of height $n$ and hence a complete intersection.
 Since all pure power survival complexes are generated by rings $R = \frac{k[x_1, \ldots, x_n]}{(x_{1}^{a_1}, \ldots, x_{n}^{a_n})}$, the above result follows from Theorem~\ref{Gore-oneSoc}.\\
\end{proof}


\subsection{More Partial Characterizations of Survival Complexes}

One research question considered was as follows. Given the survival complex $\Sigma (S)$ of $S = \frac{k[x_1, \ldots, x_n]}{I}$, if one assume some properties of the isolated points of $\Sigma(S)$, can anything be said about $\Sigma(S)$ or $S$?\\
Some results obtained are presented below.\\ 

\begin{lemma} \label{sets-fiber}
Let $\Sigma(S)$ be a survival complex with $m$ isolated points which can be separated into $t$ sets $\mathcal{H}_1, \ldots, \mathcal{H}_t$ such that for all $i$ and $j$
if $h_j \in \mathcal{H}_j$ and $h_i \in \mathcal{H}_i$ then $h_j$ and $h_i$ share no variables in common. Then $S$ can written as the fibre product of $t$ rings, $S = R_1 \circ \cdots \circ R_t$ and $\Sigma (S)$ is the disjoint union of $\Sigma (R_1), \ldots, \Sigma (R_t)$.
\end{lemma}

\begin{proof}
Choose two arbitrary sets $\mathcal{H}_i$ and $\mathcal{H}_j$ from $\mathcal{H}_1, \ldots, \mathcal{H}_t$ and consider the two subcomplexes generated from the variables in $\mathcal{H}_i$ and $\mathcal{H}_j$.\\
Since $\mathcal{H}_i$ and $\mathcal{H}_j$ share no variables in common, no monomials in either one share a common factor other than $1$.\\
Since they lack a common factor, there is no connected subcomplex which contains both.\\
Therefore the variables in $\mathcal{H}_i$ and $\mathcal{H}_j$ form two distinct subcomplexes of $\Sigma (S)$ each of which contains at least a nontrivial or quasi-trivially connected component.\\
Since we have $t$ distinct sets of variables, $\Sigma (S)$ must contain $t$ nontrivial or quasi-trivially connected components.\\
Thus by Proposition~\ref{Components-fibre}, $\Sigma (S)$ can be written as the fibre product of $t$ rings, $S = R_1 \circ \cdots \circ R_t$, where $\Sigma (S)$ is the disjoint union of $\Sigma (R_1), \ldots, \Sigma (R_t)$.\\

\end{proof}

\begin{remark}
While Proposition~\ref{Components-fibre} and Lemma~\ref{sets-fiber} are similar, Proposition~\ref{Components-fibre} allows us to decompose a ring into a fibre product using only information about its components, whereas Lemma~\ref{sets-fiber} allows us to decompose a ring into a fibre product using only information about the truly isolated points of its survival complex, i.e. the generators of its socle.\\
 
\end{remark}
 
\begin{theorem}
Let $\Sigma(S)$ be a survival complex with exactly $m$ truly isolated points, no two of which share variables in common. Then $S$ can be decomposed into the fibre product of $m$ Pure Power rings.

\end{theorem}

\begin{proof}

Let $\Sigma(S)$ be a survival complex with $m$ truly isolated points, no two of which share variables in common.\\
Then the isolated points of $\Sigma (S)$ can be divided into $m$ sets. Since no two truly isolated points share any variables in common then by Lemma~\ref{sets-fiber} we know that $S$ can be decomposed into the fibre product of $m$ rings where the survival complex of each ring corresponds to a subcomplex of $S$.\\

Since no truly isolated points share a common variable, the pairwise products of all the variables represented in different truly isolated points must be in the defining ideal of $S$.\\
To see this, let $q$ and $p$ be truly isolated points of $\Sigma(S)$ with $\lbrace x_1, \ldots, x_t \rbrace$ being the variables in $q$ and $\lbrace y_1, \ldots, y_w \rbrace$ being the variables in $p$.\\
Assume that $x_iy_j$ is not in the defining ideal of $S$.\\
Then $x_iy_j$ must be in $\Sigma (S)$ and hence by Proposition 4 is a factor of some truly isolated point of $\Sigma (S)$.\\
Therefore there exists a truly isolated point $g$ of $\Sigma (S)$ which shares variables with both $p$ and $q$, a contradiction.\\
Therefore, by Proposition 7, $\Sigma (S)$ must contain $m$ nontrivial or quasi-trivially connected components.\\

Since each of the $m$ components corresponds to a ring in the fibre product and each component only has one isolated point, then by Theorem~\ref{Gore-oneSoc} each ring in the fibre product must be Pure Power.\\
Thus $S$ can be decomposed into the fibre product of $m$ Pure Power rings.\\


%

\end{proof}

\section[Socle Analysis and Ideal Building]{Socle Analysis and Ideal Building}

\subsection{Correspondence between Ideals and Survival Complexes}
Let $X$ be a finite set of monomials such that no two monomials in $X$ are factors of one another. One thing we wish to know is given any such set $X$, when is it possible to find a ring whose survival complex's isolated points correspond exactly to $X$?\\

\begin{proposition} If $X$ is a set of monomials in $k[x_1, \ldots, x_n]$ consisting exactly of the isolated points of some survival complex $\Sigma \left( \frac{k[x_1, \ldots, x_n]}{I} \right)$, there is at most one such ideal $I$.
\end{proposition}

\begin{proof}
Let $I,J \subset K[x_{1} \ldots, x_n] = R$ be finitely generated monomial ideals that contain pure powers of $x_i$ for all $1 \leq i \leq n$.\\
Let $X$ be simultaneously the set of isolated points of the survival complexes $\Sigma(R/I)$ and $\Sigma(R/J)$.\\
As we showed earlier, if either $I$ is a subset of $J$ or $J$ is a subset of $I$, one of the survival complexes is an induced subset of the other.\\
So assume $I \not \subset J$ and $J \not \subset I$.\\
Hence there exists a monomial $j \in J$ such that $ j \not \in I$ which implies that $0 \neq j \in R/I$ and hence $j$ is a vertex in $\Sigma(R/I)$.\\
As was proved earlier, every vertex of $\Sigma(R/I)$ must be a nonzero factor of an isolated point of $\Sigma(R/I)$.
By assumption, though, the isolated points of $\Sigma (R/I)$ are the same as the isolated points of $\Sigma (R/J)$.
 Furthermore, every factor of an isolated point must be in $\Sigma(R/J)$ which implies that $j \in \Sigma(R/J)$.\\
 Therefore $j \neq 0$ in $R/J$ and hence $j \not \in J$, a contradiction.\\
So for any such $X$, there is at most one ideal which generates $X$.\\

\end{proof}

\subsection{A Faulty Algorithm}
Now given any such set $X = \lbrace q_1, \ldots, q_t \rbrace$ it is always possible to find a ring whose survival complex's set of isolated points includes $X$ by applying the following method.\\

\begin{algorithm} \label{Alg-1}
Let $\lbrace x_1, \ldots x_r \rbrace$ be all the variables present in $X$. Start with the ring $k[x_{1}, \ldots, x_{r}]$.\\
Next, generate an ideal $I$ of $k[x_{1}, \ldots, x_{r}]$ in the following way. For any $x_s \in  \lbrace x_1, \ldots x_r \rbrace$, if $x_{s}^{a_s}$ is the highest power of $x_s$ appearing in $X$, then add $x_{s}^{a_s+1}$ to the generating set of $I$.\\
Next, for any $q_s \in X$, add $x_jq_s$ to the generating set of $I$ for all $1 \leq j \leq r$.\\
Finally, if any two variables $x_j$ and $x_s$ do not appear together in any $q_l$, add $x_jx_s$ to the generating set of the ideal.\\
\end{algorithm}
Then the set of isolated points of the survival complex $\Sigma (S)$ where $S = \frac{k[x_{1}, \ldots, x_{r}]}{I}$ will contain $X$ since for all monomials $m \in X$, $m \not \in I$ but $x_j m \in I$ for all $1 \leq j \leq r$.\\

\begin{example}
 Let $X = \lbrace x^2y^2, yz \rbrace$. Using Algorithm~\ref{Alg-1}, we start with the ring $k[x,y,z]$ and then create an ideal $I$. The first monomials we add to the generating set of $I$ are $\lbrace x^3,y^3,z^2 \rbrace$. Next we add $x^3y^2,x^2y^3,xyz,y^2z,$ and $yz^2$. However, $x^3y^2,x^2y^3$ and $yz^2$ are already generated by other elements of the generating set of $I$. So our generating set for $I$ now looks like $\lbrace x^3,y^3,z^2 , xyz,y^2z \rbrace$. Finally, since $x$ and $z$ do not appear together in any monomial of $X$, we add $xz$ to the generating set and hence can remove $xyz$.\\
So $I = (x^3,y^3,z^2 ,y^2z,xz)$ and the survival complex is $\Sigma (S)$ where $S = \frac{k[x,y,z]}{(x^3,y^3,z^2 ,y^2z,xz)}$.\\

\end{example}
\begin{figure}[h]

\begin{center}

\begin{tikzpicture}
[scale=5, vertices/.style={draw, fill=black, circle, inner sep=1.8pt}]

\node[vertices] (y) at (-0.5,0) {};
\node[left] at (-0.5,0) {$y$};
\node[vertices] (x) at (0.5,0) {};
\node[right] at (0.5,0) {$x$};

\node[vertices] (z) at (-0.8,-0.3) {};
\node[below] at (-0.8,-0.3) {$z$};
\node[vertices] (yz) at (-0.8,0.3) {};
\node[right] at (-0.8,0.3) {$yz$};

\node[vertices] (xy) at (0,-0.5) {};
\node[below] at (0,-0.5) {$xy$};
\node[vertices] (yy) at (0.5,0.3) {};
\node[right] at (0.5,0.3) {$y^2$};
\node[vertices] (xx) at (-0.5,0.3) {};
\node[left] at (-0.5,0.3) {$x^2$};

\node[vertices] (xxy) at (-0.5,-0.5) {};
\node[left] at (-0.5,-0.5) {$x^2y$};
\node[vertices] (xxyy) at (0,0.5) {};
\node[below] at (0,0.5) {$x^2y^2$};
\node[vertices] (xyy) at (0.5,-0.5) {};
\node[right] at (0.5,-0.5) {$xy^2$};

\filldraw[fill=black!20, draw=black] (0,-0.5)--(-0.5,0)--(0.5,0)--cycle;

\draw (xyy)--(x);
\draw (yy)--(x);
\draw (yy)--(xx);
\draw (y)--(xx);
\draw (y)--(xxy);
\draw (y)--(xxy);
\draw (y)--(z);

\end{tikzpicture}

\caption{Survival Complex of $R = \frac{k[x,y,z]}{(x^3,y^3,z^2 ,y^2z,xz)}$}

\end{center}

\end{figure}
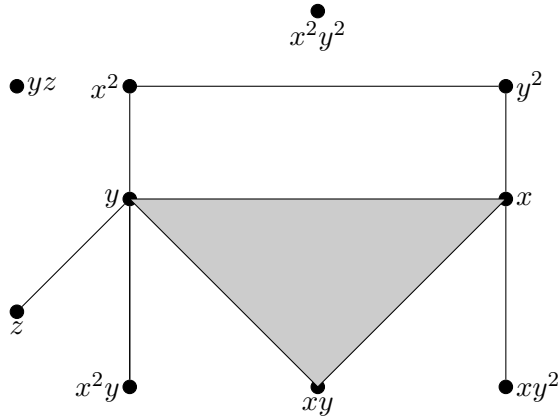

\begin{remark}
In the previous example,  Algorithm~\ref{Alg-1} produced a ring whose survival complex's isolated points were exactly $X$. This does not always happen though as the next example will show.\\

\end{remark}

\begin{example}
Let $X =\lbrace x^2y^2, x^3,y^4 \rbrace$. Using Algorithm~\ref{Alg-1}, we create an ideal $I$ of the ring $k[x,y]$ as follows. First, add $x^4$ and $y^5$ to the generating set of $I$. Next, add $xy^4, x^3y^2, x^2y^3$, and $x^3y$ to the generating set of $I$. Hence we can write $I = (x^4,y^5, x^3y, x^2y^3, xy^4)$.\\
This then gives the ring $S = \frac{k[x,y]}{(x^4,y^5, x^3y, x^2y^3, xy^4)}$.\\
$\Sigma (S)$, however, has four isolated points as is shown in the figure below.\\
\end{example}

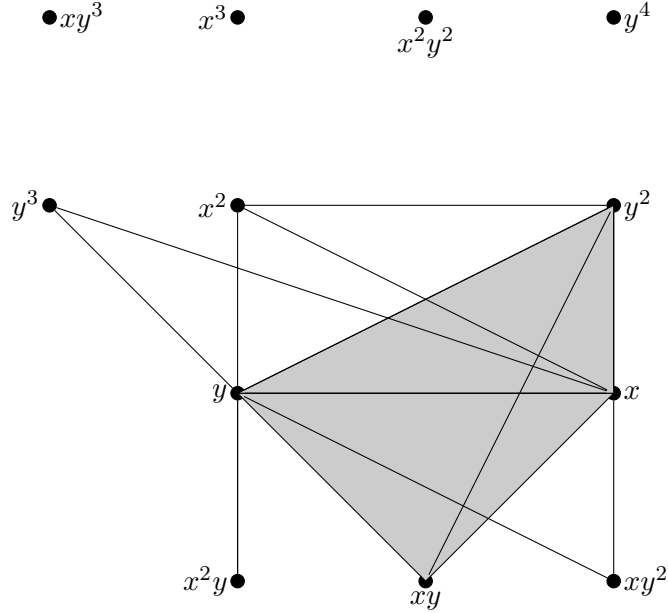
\begin{figure}[h]

\begin{center}

\begin{tikzpicture}
[scale=5, vertices/.style={draw, fill=black, circle, inner sep=1.8pt}]

\node[vertices] (y) at (-0.5,0) {};
\node[left] at (-0.5,0) {$y$};
\node[vertices] (x) at (0.5,0) {};
\node[right] at (0.5,0) {$x$};

\node[vertices] (xy) at (0,-0.5) {};
\node[below] at (0,-0.5) {$xy$};
\node[vertices] (yy) at (0.5,0.5) {};
\node[right] at (0.5,0.5) {$y^2$};
\node[vertices] (xx) at (-0.5,0.5) {};
\node[left] at (-0.5,0.5) {$x^2$};

\node[vertices] (xxy) at (-0.5,-0.5) {};
\node[left] at (-0.5,-0.5) {$x^2y$};
\node[vertices] (xxyy) at (0,1) {};
\node[below] at (0,1) {$x^2y^2$};
\node[vertices] (xyy) at (0.5,-0.5) {};
\node[right] at (0.5,-0.5) {$xy^2$};

\node[vertices] (xxx) at (-0.5,1) {};
\node[left] at (-0.5,1) {$x^3$};
\node[vertices] (yyyy) at (0.5,1) {};
\node[right] at (0.5,1) {$y^4$};
\node[vertices] (yyy) at (-1,0.5) {};
\node[left] at (-1,0.5) {$y^3$};
\node[vertices] (xyyy) at (-1,1) {};
\node[right] at (-1,1) {$xy^3$};

\filldraw[fill=black!20, draw=black] (0,-0.5)--(-0.5,0)--(0.5,0)--cycle;
\filldraw[fill=black!20, draw=black] (0.5,0.5)--(-0.5,0)--(0.5,0)--cycle;

\draw (y)--(x);
\draw (xyy)--(x);
\draw (yy)--(x);
\draw (yy)--(xx);
\draw (y)--(xx);
\draw (y)--(xxy);
\draw (y)--(xxy);

\draw (y)--(yyy);
\draw (x)--(xx);
\draw (y)--(yy);
\draw (y)--(xyy);
\draw (xy)--(yy);
\draw (x)--(yyy);

\end{tikzpicture}

\caption{Survival Complex of $R = \frac{k[x,y]}{(x^4,y^5,x^3y ,x^2y^3,xy^4)}$}

\end{center}

\end{figure}

\vspace{2 cm}

\subsection{A Two Dimensional Algorithm}

We will now consider just the two dimensional case.
Given a finite set $X$ of monomials in $x$ and $y$ such that no two monomials in $X$ are factors of one another,
 we will apply the following algorithm to get a ring $R$ such that the truly isolated points of $\Sigma (R)$ are precisely $X$.\\

\begin{algorithm} \label{Alg-2}
Since we are in the two dimensional case, we start with a polynomial ring $k[x,y]$ and construct an ideal $I$ as follows.\\

Let $x^t$ and $y^s$ be the highest powers of $x$ and $y$ which appear in $X$. Then add $x^{t+1}$ and $y^{s+1}$ to the generators of $I$.\\
Now take the element of $X$ which contains $y^s$, say $x^{a_1}y^s$, and the element which contains the next highest power of $y$, say $x^{a_2}y^{b_2}$. Then add $x^{a_1+1}y^{b_2 + 1}$ to the generators of $I$.\\
Repeat this process while comparing $x^{a_2}y^{b_2}$ to the element of $X$ with the next highest power of $y$.\\
Continue until there are no more elements of $X$.\\
\end{algorithm}

\begin{example}
We return to the set $X =\lbrace x^2y^2, x^3,y^4 \rbrace$ from Example 5. Using Algorithm~\ref{Alg-2}, we create an ideal $I$ of the ring $k[x,y]$ as follows. First, add $x^4$ and $y^5$ to the generating set of $I$. Next, note that $y^4$ contains the largest power of $y$ and $x^2y^2$ contains the second largest power of $y$.
Following Algorithm~\ref{Alg-2}, add $xy^3$ to the generating set of $I$.\\
Next, consider $x^2y^2$ and $x^4$. Following Algorithm~\ref{Alg-2}, add $x^3y$ to the generating set of $I$.\\
  Hence we can write $I = (x^4,y^5, x^3y, xy^3)$.\\
This then gives the ring $R = \frac{k[x,y]}{(x^4,y^5, x^3y, xy^3)}$.\\
Then as is shown in the figure below $\Sigma (R)$ has three isolated points which are precisely the monomials in $X$.\\

\end{example}

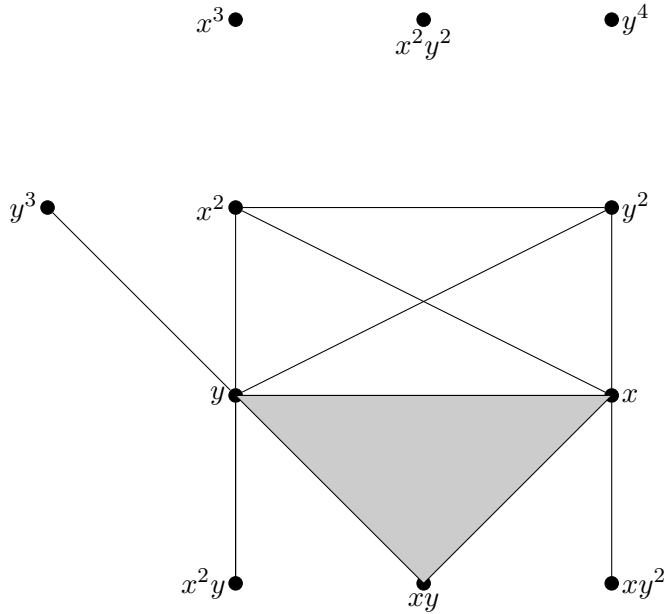
\begin{figure}[h]

\begin{center}
\begin{tikzpicture}
[scale=5, vertices/.style={draw, fill=black, circle, inner sep=1.8pt}]

\node[vertices] (y) at (-0.5,0) {};
\node[left] at (-0.5,0) {$y$};
\node[vertices] (x) at (0.5,0) {};
\node[right] at (0.5,0) {$x$};

\node[vertices] (xy) at (0,-0.5) {};
\node[below] at (0,-0.5) {$xy$};
\node[vertices] (yy) at (0.5,0.5) {};
\node[right] at (0.5,0.5) {$y^2$};
\node[vertices] (xx) at (-0.5,0.5) {};
\node[left] at (-0.5,0.5) {$x^2$};

\node[vertices] (xxy) at (-0.5,-0.5) {};
\node[left] at (-0.5,-0.5) {$x^2y$};
\node[vertices] (xxyy) at (0,1) {};
\node[below] at (0,1) {$x^2y^2$};
\node[vertices] (xyy) at (0.5,-0.5) {};
\node[right] at (0.5,-0.5) {$xy^2$};

\node[vertices] (xxx) at (-0.5,1) {};
\node[left] at (-0.5,1) {$x^3$};
\node[vertices] (yyyy) at (0.5,1) {};
\node[right] at (0.5,1) {$y^4$};
\node[vertices] (yyy) at (-1,0.5) {};
\node[left] at (-1,0.5) {$y^3$};

\filldraw[fill=black!20, draw=black] (0,-0.5)--(-0.5,0)--(0.5,0)--cycle;

\draw (y)--(x);
\draw (xyy)--(x);
\draw (yy)--(x);
\draw (yy)--(xx);
\draw (y)--(xx);
\draw (y)--(xxy);
\draw (y)--(xxy);

\draw (y)--(yyy);
\draw (x)--(xx);
\draw (y)--(yy);

\end{tikzpicture}

\caption{Survival Complex of $R = \frac{k[x,y]}{(x^4,y^5,x^3y ,xy^3)}$}

\end{center}

\end{figure}

\begin{proposition}
Let $X$ be a finite set of monomials in $x$ and $y$ such that no two monomials in $X$ are factors of one another.\\
Then it is always possible to find a ring $R$ whose truly isolated points are exactly $X$.
\end{proposition}

\begin{proof}
Consider $R = \frac{k[x,y]}{I}$ where $I$ is generated as described by Algorithm 2.\\
Then $\Sigma(R)$ clearly contains $X$ in its set of isolated points.

Assume that $\Sigma(R)$ has a truly isolated point $m$ which is not contained in $X$. 
Hence $xm \in I$ which means that for some generator $x^{a_1}y^{b_1}$, $xm = qx^{a_1}y^{b_1}$ where $q = x^{a_2}y^{b_2}$.\\
But then $m = x^{a_1 + a_2 -1}y^{b_1 + b_2}$.\\
If $a_2 > 0$, then $x^{a_1}y^{b_1}$ would be a factor of $m$ which would imply that $m \in I$, a contradiction.\\
Thus $m = x^{a_1-1}y^{b_1 +b_2}$.\\
Since Algorithm~\ref{Alg-2} was used to generate $I$, however, there exists a monomial $g \in X$ where $g = x^{a_1 -1}y^t$.\\

Since we assumed that $m \not \in X$, then either $b_1 +b_2 > t$ or $b_1 +b_2 < t$.\\
If $b_1 +b_2 > t$, then $g$ divides $m$ which implies that $g$ is not isolated, a contradiction.\\
If $b_1 + b_2 < t$, then $m$ divides $g$ which implies that $m$ is not isolated.\\

Therefore, $m$ can not exist.\\
Hence the truly isolated points of $\Sigma (R)$ are precisely $X$.\\

\end{proof}

\begin{remark}
Notice that Algorithm~\ref{Alg-2} allows for the creation of quasi-isolated points as well. If you wish to look at only those sets $X$ where $\Sigma (S)$ of $S$ generated by Algorithm~\ref{Alg-2} has no quasi-isolated points, simply define $X$ as follows.\\
Let $X$ be a finite set of monomials in $x$ and $y$ such that no two monomials in $X$ are factors of one another and $x^2$ (respectively $y^2$) is not in $X$ unless $x$ (respectively $y$) is a factor of another monomial in $X$.\\
By restricting $X$ this way, it forces there to always be an edge from $x$ and $y$ to another monomial, meaning that neither $x$ nor $y$ can be quasi-isolated.\\

\end{remark}

\subsection{Extension to a Specific Three Dimensional Case}
In this section we will extend Algorithm~\ref{Alg-2} to a specific three dimensional case. Let $X$ be a finite set of monomials in $x_1$, $x_2$, and $x_3$ such that $X = x_{i}^lY$, $1 \leq i \leq 3$, where $Y \subset k[x_j,x_f]$, $1 \leq j,f \leq 3$, $j \neq f \neq i$, is a finite set of monomials such that no two monomials are factors of each other.\\
We will apply the following algorithm to get a ring $R$ such that the truly isolated points of $\Sigma (R)$ are precisely $X$.\\

\begin{algorithm} \label{Alg-3}
Start with the polynomial ring $k[x_1,x_2,x_3]$ and construct an ideal $I$ as follows. 

Denote by $z$ the $x_i$ such that $X = x_{i}^lY$. 
Denote by $x$ and $y$ the two remaining $x_j$'s.\\
First add $z^{l+1}$ to the generators of $I$.
Now let $x^t$ and $y^s$ be the highest powers of $x$ and $y$ which appear in $X$. Then add $x^{t+1}$ and $y^{s+1}$ to the generators of $I$.\\
Now take the element of $X$ which contains $y^s$, say $x^{a_1}y^sz^l$, and the element which contains the next highest power of $y$, say $x^{a_2}y^{b_2}z^l$. Then add $x^{a_1+1}y^{b_2 + 1}$ to the generators of $I$.\\
Repeat this process while comparing $x^{a_2}y^{b_2}z^l$ to the element of $X$ with the next highest power of $y$.\\
Continue until there are no more elements of $X$.\\

\end{algorithm}

\begin{remark}
Note that since every element of $X$ is of the form $x^ay^bz^l$ in order for no two points to be factors of each other the degree of $y$ must decrease as the degree of $x$ increases.\\
\end{remark}

\begin{example}
Let $X =\lbrace x_{1}^2x_{2}^2, x_{2}^2x_{3}^2 \rbrace$. Using Algorithm~\ref{Alg-3}, we create an ideal $I$ of the ring $k[x_1,x_2,x_3]$ as follows. First, add $x_{1}^3$, $x_{2}^3$ and $x_{3}^3$ to the generating set of $I$. Next, note that all elements of $X$ contain $x_{2}^2$ which means that we will only be adding elements of the form $x_{1}^ax_{3}^b$ to $I$.
Since $X$ contains only two monomials, by Algorithm~\ref{Alg-3} the only monomial we add to the generators of $I$ is $x_{1}x_{3}$\\
Hence we can write $I = (x_{1}^3,x_{2}^3,x_{3}^3,x_1x_3)$.\\
This then gives the ring $R = \frac{k[x_1,x_2,x_3]}{(x_{1}^3,x_{2}^3,x_{3}^3,x_1x_3)}$.\\
Then as is shown in the figure below $\Sigma (R)$ has two truly isolated points which are precisely the monomials in $X$.\\

\end{example}

\begin{figure}[h]

\begin{center}

\begin{tikzpicture}
[scale=5, vertices/.style={draw, fill=black, circle, inner sep=1.8pt}]

\node[vertices] (y) at (0,0) {};
\node[above] at (0,0) {$x_2$};
\node[vertices] (x) at (-1,0) {};
\node[left] at (-1,0) {$x_1$};
\node[vertices] (z) at (1,0) {};
\node[right] at (1,0) {$x_3$};

\node[vertices] (xy) at (-1,-0.5) {};
\node[below] at (-1,-0.5) {$x_1x_2$};
\node[vertices] (xx) at (-0.6,-0.5) {};
\node[below] at (-0.6,-0.5) {$x_{1}^2$};
\node[vertices] (xxy) at (-0.4,-0.5) {};
\node[below] at (-0.4,-0.5) {$x_{1}^2x_2$};

\node[vertices] (yz) at (1,-0.5) {};
\node[below] at (1,-0.5) {$x_2x_3$};
\node[vertices] (zz) at (0.6,-0.5) {};
\node[below] at (0.6,-0.5) {$x_{3}^2$};
\node[vertices] (yzz) at (0.4,-0.5) {};
\node[below] at (0.4,-0.5) {$x_2x_{3}^2$};

\node[vertices] (yy) at (0,0.5) {};
\node[above] at (0,0.5) {$x_{2}^2$};

\node[vertices] (xxyy) at (-1,0.5) {};
\node[above] at (-1,0.5) {$x_{1}^2x_{2}^2$};
\node[vertices] (yyzz) at (1,0.5) {};
\node[above] at (1,0.5) {$x_{2}^2x_{3}^2$};

\filldraw[fill=black!20, draw=black] (1,0)--(0,0)--(1,-0.5)--cycle;
\filldraw[fill=black!20, draw=black] (-1,0)--(0,0)--(-1,-0.5)--cycle;

\draw (y)--(x);
\draw (y)--(z);
\draw (yy)--(x);
\draw (yy)--(z);

\draw (y)--(xy);
\draw (y)--(xx);
\draw (y)--(xxy);
\draw (y)--(yz);
\draw (y)--(zz);
\draw (y)--(yzz);

\draw (yy)--(x);
\draw (yy)--(z);

\draw (z)--(yz);
\draw (x)--(xy);

\end{tikzpicture}

\caption{Survival Complex of $R = \frac{k[x_1,x_2,x_3]}{(x_{1}^3,x_{2}^3,x_{3}^3 , x_1x_3)}$}

\end{center}

\end{figure}
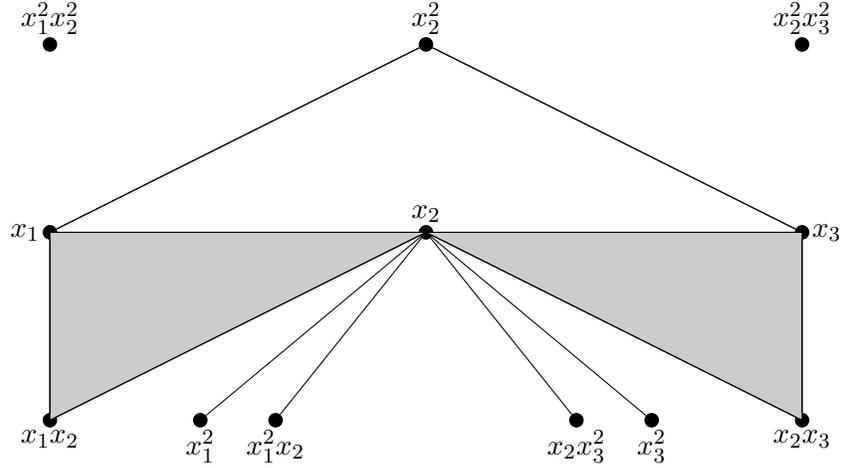

\begin{proposition}
Let $X$ be a finite set of monomials in $x_1$, $x_2$, and $x_3$ such that $X = x_{i}^lY$, $1 \leq i \leq 3$, where $Y \subset k[x_j,x_f]$, $1 \leq j,f \leq 3$, $j \neq f \neq i$, is a finite set of monomials such that no two monomials are factors of each other.\\
Then it is always possible to find a ring $R$ whose truly isolated points are exactly $X$.
\end{proposition}

\begin{proof}
Consider $R = \frac{k[x_1,x_2,x_3]}{I}$ where $I$ is the ideal constructed in Algorithm~\ref{Alg-3}.
We will denote by $z$ the $x_i$ such that $X = x_{i}^lY$. Denote by $x$ and $y$ the other two $x_j$'s.\\
Clearly $X$ is contained in the set of isolated points of $\Sigma (R)$.\\

Assume that $\Sigma(R)$ has a truly isolated point $m$ which is not contained in $X$.
Then $m = x^{a_m}y^{b_m}z^c$ where $0 \leq a_m$, $0 \leq b_m$, and $0 \leq c \leq l$.

Since $m $ is truly isolated we can write that $xm =qt$ for some generator $t$ of $I$. From our algorithm, we know that only generator of $I$ to contain $z$ is $z^{l+1}$. Since $c \leq l$, $t$ cannot contain $z$ as a factor.\\
Therefore we can write $x(x^{a_m}y^{b_m}z^c) = q(x^{a_1}y^{b_1})$ where $x^{a_1}y^{b_1}$ is a generator of $I$. This implies that $q = x^{a_2}y^{b_2}z^c$ and that $a_m = a_1 +a_2 -1$ while $b_m = b_1 +b_2$.\\
If $a_2 > 0$, then $x^{a_1}y^{b_1}$ would be a factor of $m$ which would imply that $m \in I$, a contradiction.\\
Thus $m = x^{a_1-1}y^{b_1 +b_2}z^c$.\\
However, by the algorithm we used to generate $I$, there exists a monomial $g \in X$ of the form $g = x^{a_1 -1}y^dz^l$.\\

Since we assumed that $m \not \in X$, and we know that $c \leq l$ then either $b_1 +b_2 > d$ or $b_1 +b_2 < d$.\\
If $b_1 +b_2 > d$, then $g$ divides $m$ which implies that $g$ is not isolated, a contradiction.
If $b_1 + b_2 < d$, then $m$ divides $g$ which implies that $m$ is not isolated.
Therefore, $m$ can not exist.\\
Hence the truly isolated points of $\Sigma (R)$ are precisely $X$.\\

\end{proof}



\section*{Acknowledgments}

My sincerest thanks to my advisor, Neil Epstein at George Mason University. Without his support and mathematical knowledge this paper would never have been possible.

\bibliographystyle{amsalpha} 
\bibliography{WolffThesisbib}

\end{document}